\documentclass[12pt,a4paper]{article}
\usepackage{amsmath,amssymb,amsthm,bm}
\usepackage{graphicx}
\usepackage{txfonts}

\newtheorem{theorem}{Theorem}[section]
\newtheorem{proposition}[theorem]{Proposition}
\newtheorem{corollary}[theorem]{Corollary}

\newtheorem{remark}[theorem]{Remark}
\newtheorem{example}[theorem]{Example}

\newcommand{\1}{{\bm 1}}

\numberwithin{equation}{section}

\begin{document}

\begin{center}
\large\bf
Quadratic Embedding Constants of Strongly Regular Graphs
\end{center}

\bigskip

\begin{center}
Nobuaki Obata\\
Center for Data-driven Science and Artificial Intelligence \\
Tohoku University\\
Sendai 980-8576 Japan \\
obata@tohoku.ac.jp
\end{center}

\bigskip

\begin{quote}
\textbf{Abstract}\enspace
We obtain an explicit formula for the 
quadratic embedding constant (QEC) of a strongly regular graph
$\mathrm{srg}(n,k,\lambda,\mu)$ with $\mu\ge1$.
By using QEC we give a necessary and sufficient condition 
for a strongly regular graph to admit a quadratic embedding
in a Euclidean space.
\end{quote}

\begin{quote}
\end{quote}

\begin{quote}
\end{quote}

\section{Introduction}
\label{01sec:Introduction}

Let $G=(V,E)$ be a finite connected graph on $n=|V|\ge2$ vertices,
and $D=[d(x,y)]_{x,y\in V}$ the distance matrix of $G$,
where $d(x,y)$ is the standard graph distance
defined as the length of a shortest
walk (or path) connecting $x$ and $y$.
Let $C(V)$ denote the space of real column vectors
$f=[f_x]_{x\in V}$ indexed by $V$ and
$\langle\cdot,\cdot\rangle$ the canonical inner product.
The \textit{quadratic embedding constant} 
(\textit{QE constant} for short) of $G$ is defined to be
the conditional maximum of
the quadratic function $\langle f,Df \rangle$,
$f\in C(V)$, subject to two constraints
$\langle f,f \rangle=1$ and $\langle \1,f \rangle=0$,
where $\bm{1}\in C(V)$ is the column vector whose entries are all one.
In short,
\begin{equation}\label{01eqn:def of QEC}
\mathrm{QEC}(G)
=\max\{\langle f,Df \rangle\,;\, f\in C(V), \,
\langle f,f \rangle=1, \, \langle \1,f \rangle=0\}.
\end{equation}
The QE constant was first introduced
in \cite{Obata-Zakiyyah2018}
keeping a profound relation to Euclidean distance geometry
and quantum probability.

A \textit{quadratic embedding} of 
a connected graph $G=(V,E)$ is 
a map $\psi:V\rightarrow \mathbb{R}^N$ satisfying
\[
\|\psi(x)-\psi(y)\|^2=d(x,y),
\qquad x,y\in V,
\]
where the left-hand side is the square of the Euclidean distance
and the right-hand side is the graph distance.
This concept is fundamental in Euclidean distance geometry
\cite{Balaji-Bapat2007,Deza-Laurent1997,
Jaklic-Modic2013,Jaklic-Modic2014,
Liberti-Lavor-Maculan-Mucherino2014},
tracing back to Schroenberg \cite{Schoenberg1935,Schoenberg1938}.
Moreover, it gives a criterion for a state associated with the
$Q$-matrix $Q=[q^{d(x,y)}]$ to be positive definite,
which is important in harmonic analysis and quantum probability,
see \cite{Hora-Obata2007} and references cited therein.
Here, as a consequence of Schoenberg's theorem,
we note that a graph $G=(V,E)$ admits a quadratic embedding
if and only if $\mathrm{QEC}(G)\le0$.
Such a graph is called \textit{of QE class}.
Graphs of both QE class and non-QE class are interesting
from various aspects.

In these years, 
beyond the original motivation 
towards quantitative approach to quadratic embedding of graphs,
there is growing interest in the QE constant
as a new numeric invariant of connected graphs.
For some classes of graphs, the QE constants are obtained explicitly,
for example, graphs on five or less vertices
\cite{Obata-Zakiyyah2018},
complete graphs and cycles \cite{Obata-Zakiyyah2018},
complete multipartite graphs \cite{Obata2023b},
double star graphs \cite{Choudhury-Nandi2023},
fan graphs \cite{MO-2024a, MO-2024b},
hairy cycle graphs \cite{Irawan-Sugeng2021},
paths \cite{Mlotkowski2022},
wheel graphs \cite{Obata2017}, and so on.
Moreover, formulas and estimates are investigated
in relation to some graph operations 
\cite{Choudhury-Nandi2023,Lou-Obata-Huang2022,MO-2018}.
It is a challenging problem to classify graphs
by using the QE constants,
for example, an attempt at classification of graphs
with $\mathrm{QEC}(G)<-1/2$ was initiated 
in \cite{Baskoro-Obata2021, Baskoro-Obata2024},
where graphs with certain block structures emerge.
On the other hand, it is also a highly non-trivial problem to
explore non-QE graphs and their QE constants
\cite{MSW2024,Obata2023a}.

With the above-mentioned backgrounds, in this short paper
we obtain an explicit formula for 
the QE constant of a strongly regular graph.
In fact, with the help of the well-known result
on the eigenvalues of the adjacency matrix and the fact that 
$\mathrm{QEC}(G)$ coincides with the second largest eigenvalue
of the distance matrix for a transmission regular graph $G$,
we will prove the following results.

\begin{theorem}\label{01thm:main formula}
Let $G=\mathrm{srg}(n,k,\lambda,\mu)$ be 
a strongly regular graph with $\mu\ge1$.
Then we have
\begin{equation}\label{01eqn:main formula}
\mathrm{QEC}(G)
=-2-\frac{\lambda-\mu-\sqrt{(\lambda-\mu)^2+4(k-\mu)}}{2}\,.
\end{equation}
\end{theorem}

\begin{theorem}\label{01thm:QEC le 0}
Let $G=\mathrm{srg}(n,k,\lambda,\mu)$ be 
a strongly regular graph with $\mu\ge1$.
Then $\mathrm{QEC}(G)\le0$ if and only if $k-2\lambda+\mu\le 4$.
Moreover, 
$\mathrm{QEC}(G)=0$ if and only if $k-2\lambda+\mu=4$.
\end{theorem}

\begin{theorem}\label{01thm:QEC<0}
For any strongly regular graph
$G=\mathrm{srg}(n,k,\lambda,\mu)$ with $\mu\ge1$,
we have $\mathrm{QEC}(G)\ge0$ except $G=C_5=\mathrm{srg}(5,2,0,1)$,
and 
\[
\mathrm{QEC}(C_5)=-\frac{3-\sqrt{5}}{2}<0.
\]
\end{theorem}

The paper is organized as follows.
In Section \ref{02sec:Strongly Regular Graphs} 
we assemble basic notions and notations
of strongly regular graphs,
and in Section \ref{03sec:Distance Matrices}
we recall some properties of their adjacency matrices.
In Section \ref{04sec:Distance Matrices} we obtain the 
eigenvalues of the distance matrix
in terms of those of the adjacency matrix.
In Section \ref{05sec:Quadratic Embedding Constants}
we prove the main results
(Theorems \ref{01thm:main formula}--\ref{01thm:QEC<0}).
In Section \ref{06sec:Examples} we show some examples.

\section{Strongly Regular Graphs}
\label{02sec:Strongly Regular Graphs}

To avoid ambiguity, we begin with some basic notions and notations.
A \textit{graph} $G=(V,E)$ is a pair of a finite non-empty set $V$
and a set $E$ of two-element subsets of $V$.
As usual, an element of $V$ is called a \textit{vertex}
and one of $E$ an \textit{edge}.
Two vertices $x,y\in V$ are called \textit{adjacent} if
$\{x,y\}\in E$, and in that case we write $x\sim y$. 
We also say that 
$y\in V$ is a \textit{neighbour} of $x\in V$ if $x\sim y$.

A graph $G=(V,E)$ is called \textit{regular} if 
every vertex has the same number of neighbors,
i.e.,
\begin{equation}\label{01eqn:degree of G}
k=|\{y\in V\,;\, y\sim x\}|
\end{equation}
is constant, regardless of the choice of $x\in V$.
In that case, $k$ is called the \textit{degree} of $G$,
and $G$ is also called a \textit{$k$-regular graph}.
A regular graph $G=(V,E)$ is called \textit{strongly regular} if 
\begin{description}
\item[(SR1)] every two adjacent vertices have the 
same number of common neighbours, i.e.,
\begin{equation}\label{01eqn:def of lambda}
\lambda=|\{z\in V\,;\, z\sim x, \,\, z\sim y\}|
\end{equation}
is constant, regardless of the choice of $x,y\in V$
such that $x\sim y$;
\item[(SR2)] every two non-adjacent vertices have the 
same number of common neighbours, i.e.,
\begin{equation}\label{01eqn:def of mu}
\mu=|\{z\in V\,;\, z\sim x, \,\, z\sim y\}|
\end{equation}
is constant, regardless of the choice of $x,y\in V$
such that $x\not\sim y$ and $x \neq y$.
\end{description}
Those numbers appearing in conditions (i) and (ii)
together with $n=|V|$ and the degree $k$
are characteristics of a strongly regular graph $G$,
and we write $G=\mathrm{srg}(n,k,\lambda,\mu)$ for convenience.
It is noted, however, that the parameters
$n,k,\lambda$ and $\mu$ do
not specify a strongly regular graph uniquely
(up to graph isomorphisms).

\begin{remark}\label{02rem:empty graphs}
\normalfont
Condition (i) is fulfilled 
if there is no pair of adjacent vertices,
though the number $\lambda$ is not determined.
This occurs if and only if $G$ is
an empty graph $\bar{K}_n$ with $n\ge1$. 
\end{remark}

\begin{remark}\label{02rem:complete graphs}
\normalfont
Condition (ii) is fulfilled 
if there is no pair of non-adjacent vertices,
though the number $\mu$ is not determined.
This occurs if and only if 
any pair of distinct vertices are adjacent,
that is, $G$ is a complete graph $K_n$ with $n\ge1$.
\end{remark}

It is convenient to refer to a general strongly regular graph
$G$ with its parameter $\mathrm{srg}(n,k,\lambda,\mu)$
even though 
$\lambda$ is not determined (Remark \ref{02rem:empty graphs})
or 
$\mu$ is not determined (Remark \ref{02rem:complete graphs}).
In this connection, we note the following result
which is verified easily by definition.

\begin{proposition}\label{02prop:mu=0}
For a strongly regular graph $G=(V,E)$ the following three conditions
are equivalent:
\begin{enumerate}
\setlength{\itemsep}{0pt}
\item[\upshape (i)] $G$ is disconnected;
\item[\upshape (ii)] $\mu=0$
(this presupposes that $\mu$ is determined);
\item[\upshape (iii)] $G=pK_q$ with $p\ge2$ and $q\ge1$;
\end{enumerate}
\end{proposition}

First of all, the QE constant is defined only for
a connected graph on two or more vertices,
see Introduction.
Moreover, the QE constant of the complete graph $K_n$ is 
known as $\mathrm{QEC}(K_n)=-1$ for $n\ge2$, 
see \cite{Obata-Zakiyyah2018}.
Thus, our target is the class of strongly regular graphs
which are connected and are not complete graphs.
As is mentioned in Remark \ref{02rem:complete graphs},
the strongly regular graphs $\mathrm{srg}(n,k,\lambda,\mu)$
with undetermined $\mu$ are the complete graphs.
It follows from Proposition \ref{02prop:mu=0} that
the strongly regular graphs $\mathrm{srg}(n,k,\lambda,\mu)$
which are not connected are characterized as $\mu=0$.
Therefore, the class of strongly regular graphs
$\mathrm{srg}(n,k,\lambda,\mu)$ 
which are connected and are not complete graphs
is characterized by $\mu\ge1$.

Below we list some properties of strongly regular graphs
$\mathrm{srg}(n,k,\lambda,\mu)$ with $\mu\ge1$.
The proofs are known and straightforward,
see e.g., \cite{Brouwer-Cohen-Neumaier1989,Brouwer-Haemers2012,
Brouwer-VanMaldeghem2022}.

\begin{proposition}\label{02prop:diam=2}
For a strongly regular graph $G=\mathrm{srg}(n,k,\lambda,\mu)$
with $\mu\ge1$, we have
\[
\mathrm{diam}(G)=2,
\]
where $\mathrm{diam}(G)=\max\{d(x,y)\,;\, x,y\in V\}$
is the diameter of $G$.
\end{proposition}

\begin{proposition}\label{02prop:bounds of parameters}
For a strongly regular graph $G=\mathrm{srg}(n,k,\lambda,\mu)$
with $\mu\ge1$, we have
\begin{equation}\label{02eqn:parameter conditions}
n\ge 4, \qquad
2\le k\le n-2, \qquad
1\le \mu \le k, \qquad
0\le \lambda\le k-2.
\end{equation}
\end{proposition}

\begin{proposition}\label{02prop:relation among parameters}
For a strongly regular graph $G=\mathrm{srg}(n,k,\lambda,\mu)$
with $\mu\ge1$, we have
\begin{equation}\label{02eqn:relation among parameters}
(n-k-1)\mu=(k-\lambda-1)k.
\end{equation}
\end{proposition}

\begin{remark}\normalfont
Both sides of \eqref{02eqn:relation among parameters}
do not vanish for any $G=\mathrm{srg}(n,k,\lambda,\mu)$ with $\mu\ge1$.
For a general strongly regular graph 
the relation \eqref{02eqn:relation among parameters}
is accepted as $0=0$, though $\lambda$ or $\mu$ may not be defined.
\end{remark}

\section{Adjacency Matrices}
\label{03sec:Distance Matrices}

The adjacency matrix $A$ of a graph $G=(V,E)$ is
a matrix with index set $V\times V$ defined by
\[
(A)_{xy}=\begin{cases}
1, & \text{if $x\sim y$}, \\
0, & \text{otherwise},
\end{cases}
\qquad x,y \in V.
\]
Below we list some results on the 
adjacency matrices of strongly regular graphs.
The proofs are known, see e.g.,
\cite{Brouwer-Cohen-Neumaier1989,Brouwer-Haemers2012,
Brouwer-VanMaldeghem2022}.

\begin{proposition}\label{02prop:expression A2}
Let $A$ be the adjacency matrix of a strongly regular graph 
$G=\mathrm{srg}(n,k,\lambda,\mu)$ with $\mu\ge1$.
Then we have
\begin{equation}\label{02eqn:expression A2}
A^2=\mu J-(\mu-\lambda)A-(\mu-k)I,
\end{equation}
where $J$ is the all-one matrix and $I$ the identity matrix
(their sizes are understood in the context).
\end{proposition}

\begin{proof}
Let $\bar{G}=(V, \bar{E})$ be the complement of $G=(V,E)$
and $\bar{A}$ its adjacency matrix.
Then by definition we have
\begin{equation}\label{02eqn:bar A}
\bar{A}=J-I-A.
\end{equation}
Since $G=\mathrm{srg}(n,k,\lambda,\mu)$ is a strongly regular graph,
observing that
$(A^2)_{xy}$ counts the number of 2-step walks from $x$ to $y$,
we obtain
\begin{equation}\label{02eqn:A2 by definition}
A^2=kI+\lambda A+\mu\bar{A}.
\end{equation}
Then \eqref{02eqn:expression A2} follows from
\eqref{02eqn:bar A} and \eqref{02eqn:A2 by definition}.
\end{proof}

\begin{proposition}\label{02prop:minimal polynomial of A}
Let $A$ be the adjacency matrix of 
a strongly regular graph $G=\mathrm{srg}(n,k,\lambda,\mu)$ with $\mu\ge1$.
Then we have
\begin{equation}\label{02eqn:minimal polynomial of A}
(A-k)(A^2+(\mu-\lambda)A+(\mu-k)I)=0.
\end{equation}
Moreover, the above cubic polynomial is the minimal
polynomial of $A$.
\end{proposition}

\begin{proof}
Relation \eqref{02eqn:minimal polynomial of A} follows
from Proposition \ref{02prop:expression A2} 
and $AJ=JA=kJ$.
The last statement follows 
by showing that $A^2+pA+qI\neq0$ for any $p,q\in\mathbb{R}$.
\end{proof}

\begin{proposition}\label{02prop:eigenvalues of A}
Let $G$ be a strongly regular graph $\mathrm{srg}(n,k,\lambda,\mu)$ with $\mu\ge1$.
Then the eigenvalues of its adjacency matrix are given by
$\{s,r,k\}$, where
\begin{align*}
s&=\frac{(\lambda-\mu)-\sqrt{(\lambda-\mu)^2+4(k-\mu)}}{2}\,,
\\
r&=\frac{(\lambda-\mu)+\sqrt{(\lambda-\mu)^2+4(k-\mu)}}{2}\,.
\end{align*}
Moreover, we have
\begin{equation}\label{02eqn:s<r<k}
s<r<k.
\end{equation}
\end{proposition}

\begin{proof}
The eigenvalues of $A$ are obtained from its minimal
polynomial in Proposition \ref{02prop:minimal polynomial of A}.
Since
\[
(s-r)^2
=(\lambda-\mu)^2+4(k-\mu)
=(\lambda-\mu+2)^2+4(k-1-\lambda)>0
\]
by Proposition \ref{02prop:bounds of parameters},
we obtain $s<r$.
Using $s+r=\lambda-\mu$ and $sr=\mu-k$, we have
\begin{align}
(k-s)(k-r)
&=k^2-(\lambda-\mu)k+\mu-k 
\nonumber \\
&=k(k-\lambda-1)+\mu(k+1) 
\nonumber \\
&=(n-k-1)\mu+\mu(k+1)
=n\mu>0,
\end{align}
where Proposition \ref{02prop:relation among parameters}
is taken into account.
We thus see that $k-s$ and $k-r$ have the same signature.
If $k<s$ and $k<r$,
all eigenvalues of $A$ become positive and
$\mathrm{Tr\,}A=0$ is not satisfied.
Consequently, we have $k>s$ and $k>r$,
and \eqref{02eqn:s<r<k}.
\end{proof}

\begin{proposition}\label{02prop:k,r,s are integers}
Let $G$ be 
a strongly regular graph $\mathrm{srg}(n,k,\lambda,\mu)$ with $\mu\ge1$.
If $2k+(n-1)(\lambda-\mu)\neq0$, then the three eigenvalues $k,r,s$
are all integers.
\end{proposition}

\begin{proof}
We first note that the eigenvalue $k$ of $A$ is simple.
This is proved directly or by the Peron-Frobenius theorem.
Let $f$ and $g$ be the multiplicities of the eigenvalues 
$r$ and $s$, respectively.
Since the size of $A$ is $n$ and $\mathrm{Tr\,}A=0$, we have
\[
1+f+g=n, 
\qquad
k+rf+sg=0.
\]
Solving the above linear system, we obtain
\[
f=\frac{-s(n-1)-k}{r-s}\,,
\qquad
g=\frac{r(n-1)+k}{r-s}\,,
\]
where we note that $s<r$ by Proposition \ref{02prop:eigenvalues of A}.
Using the explicit expressions of $s$ and $r$ 
in Proposition \ref{02prop:eigenvalues of A}, we obtain
\begin{equation}\label{02eqn:f-g}
f-g=\frac{2k+(n-1)(\lambda-\mu)}{\sqrt{D}}\,,
\qquad
D=(\lambda-\mu)^2+4(k-\mu).
\end{equation}
Since $f-g$ is an integer and
$2k+(n-1)(\lambda-\mu)\neq0$ by assumption, 
we see that $\sqrt{D} \in\mathbb{Q}$.
Since $D\ge1$ is an integer, so is $\sqrt{D}$.
Thus, both $r-s=\sqrt{D}$ and $r+s=\lambda-\mu$ are
integers.
It is easy to see that the parities of $r-s$ and $r+s$
coincide.
Consequently, $s$ and $r$ are integers.
\end{proof}

\begin{proposition}\label{02prop:conference graph}
Let $G$ be 
a strongly regular graph $\mathrm{srg}(n,k,\lambda,\mu)$ with $\mu\ge1$.
If $2k+(n-1)(\lambda-\mu)=0$, then
$n\equiv 1 (\mathrm{mod}\, 4)$ and 
\begin{equation}\label{02eqn:conference graph}
k=\frac{n-1}{2}\,,
\qquad
\lambda=\frac{n-5}{4}\,,
\qquad
\mu=\frac{n-1}{4}.
\end{equation}
Moreover, 
\begin{equation}\label{02eqn:eigenvalues conference graph}
s=\frac{-1-\sqrt{n}}{2}\,,
\qquad
r=\frac{-1+\sqrt{n}}{2}\,,
\end{equation}
and their multiplicities match and become $(n-1)/2$.
\end{proposition}

\begin{proof}
By assumption we have
\begin{equation}\label{02eqn:in proof prop2.11 (1)}
k=\frac{1}{2}(n-1)(\mu-\lambda).
\end{equation}
Using $2\le k\le n-2$ in Proposition \ref{02prop:bounds of parameters},
we obtain
\[
\frac{4}{n-1}\le \mu-\lambda
\le 2-\frac{2}{n-1}\,.
\]
Since $\lambda$ and $\mu$ are integers,
we obtain $n\ge5$ and $\mu-\lambda=1$.
Then $k=(n-1)/2$ by \eqref{02eqn:in proof prop2.11 (1)}
and \eqref{02eqn:conference graph} follows 
by Proposition \ref{02prop:relation among parameters}.
Finally, since $\lambda$ and $\mu$ are integers in 
\eqref{02eqn:conference graph},
it follows that $n\equiv 1 \,(\mathrm{mod}\, 4)$.
Finally, for \eqref{02eqn:eigenvalues conference graph} we only note
that $D=n$ by \eqref{02eqn:conference graph}
and use Proposition \ref{02prop:eigenvalues of A}.
\end{proof}

A strongly regular graph 
$\mathrm{srg}(n,k,\lambda,\mu)$
with parameters given as in \eqref{02eqn:conference graph} is
called a \textit{conference graph}.
However, it is not fully known 
which $n$'s have corresponding conference graphs.

\section{Distance Matrices}
\label{04sec:Distance Matrices}

\begin{proposition}\label{03prop:D=2J-2I-A}
Let $D$ be the distance matrix of a strongly regular graph
$G=\mathrm{srg}(n,k,\lambda,\mu)$ with $\mu\ge1$.
Then we have
\begin{equation}\label{03eqn:distance matrix}
D=\frac{1}{\mu}\,(2A^2-(2\lambda-\mu)A-2k I).
\end{equation}
\end{proposition}

\begin{proof}
Since $\mathrm{diam}(G)=2$ by Proposition \ref{02prop:diam=2},
we have $D=2J-2I-A$.
Then, applying Proposition \ref{02prop:expression A2},
we obtain \eqref{03eqn:distance matrix}.
\end{proof}

\begin{proposition}\label{03prop:eigen(r)<eigen(s)<eigen(k)}
Let $G=\mathrm{srg}(n,k,\lambda,\mu)$ 
be a strongly regular graph with $\mu\ge1$,
and $s<r<k$ the eigenvalues of the
adjacency matrix as in Proposition \ref{02prop:eigenvalues of A}.
Then the eigenvalues of the distance matrix are given by
\begin{equation}\label{03eqn:ev(D)}
-r-2<-s-2<2(n-1)-k.
\end{equation}
\end{proposition}

\begin{proof}
In view of \eqref{03eqn:distance matrix} we define
a polynomial by
\[
\varphi(\xi)=\frac{1}{\mu}\,(2\xi^2-(2\lambda-\mu)\xi-2k).
\]
Then by the spectral mapping theorem, 
the eigenvalues of the distance matrix $D$ are given by
$\{\varphi(s), \varphi(r), \varphi(k)\}$.
We start with $\varphi(k)$.
By definition we have
\begin{align*}
\varphi(k)
&=\frac{1}{\mu}\,(2k^2-(2\lambda-\mu)k-2k) \\
&=\frac{2}{\mu}\,k(k-\lambda-1)+k.
\end{align*}
With the help of Proposition \ref{02prop:relation among parameters}
we have
\[
\varphi(k)
=\frac{2}{\mu}\,\mu(n-k-1)+k
=2(n-1)-k.
\]
For $\varphi(r)$, we note that
$r^2=-(\mu-\lambda)r-(\mu-k)$, 
which follows by observing that
$r$ is a root of the equation
$\xi^2+(\mu-\lambda)\xi+(\mu-k)=0$.
Then we have
\begin{align*}
\varphi(r)
&=\frac{1}{\mu}\,(2r^2-(2\lambda-\mu)r-2k) \\
&=\frac{1}{\mu}\,(-2(\mu-\lambda)r-2(\mu-k)
-(2\lambda-\mu)r-2k) \\
&=-r-2.
\end{align*}
Similarly, we have $\varphi(s)=-s-2$.
Thus, the three numbers 
\[
\varphi(k)=2(n-1)-k,
\qquad
\varphi(r)=-r-2,
\qquad
\varphi(s)=-s-2
\]
are the eigenvalues of the distance matrix.

We next prove the inequalities in \eqref{03eqn:ev(D)}.
Since $s<r$ by Proposition \ref{02prop:eigenvalues of A},
we have $\varphi(r)<\varphi(s)$.
We will prove that $\varphi(s)<\varphi(k)$.
Note first that
\[
\varphi(\xi)
=\frac{2}{\mu}
\bigg(\xi-\frac{2\lambda-\mu}{4}\bigg)^2
-\frac{(\mu-2\lambda)^2+16k}{8\mu}\,,
\]
and
\[
s<\frac{2\lambda-\mu}{4}<k,
\]
which are verified directly.
Then for $\varphi(s)<\varphi(k)$ it is sufficient to show that
\begin{equation}\label{03eqn:distance from axis}
\frac{2\lambda-\mu}{4}-s<k-\frac{2\lambda-\mu}{4}\,.
\end{equation}
Using the explicit expression of $s$,
we see that \eqref{03eqn:distance from axis} is equivalent to
the following:
\begin{equation}\label{03eqn:in proof Prop3.2 (2)}
2k-\lambda>\sqrt{(\lambda-\mu)^2+4(k-\mu)}.
\end{equation}
This is verified by observing that
\[
(2k-\lambda)^2-\{(\lambda-\mu)^2+4(k-\mu)\}
=(2k-\mu)(2(k-2-\lambda)+\mu+2)+2\mu>0
\]
and both sides of \eqref{03eqn:in proof Prop3.2 (2)} are positive.
\end{proof}

\section{Quadratic Embedding Constants}
\label{05sec:Quadratic Embedding Constants}

In general, a connected graph $G=(V,E)$ 
is called \textit{transmission regular}
if its distance matrix $D$ has a constant row sum,
or equivalently if $\bm{1}\in C(V)$ is an eigenvector of $D$.
It follows from the Peron-Frobenius theorem that
the largest eigenvalue of $D$, denoted by $\delta_1(G)$, is simple
and the associated eigenspace is spanned by $\bm{1}$.
Since the eigenspaces of $D$ are mutually orthogonal,
the second largest eigenvalue $\delta_2(G)$ is obtained as
\[
\max\{\langle f,Df\rangle\,;\,f\in C(V),\,\,
\langle f,f\rangle=1,\,\,
\langle \bm{1}, f\rangle=0\},
\]
which coincides with $\mathrm{QEC}(G)$ by definition.
We thus come to the following noteworthy result.

\begin{proposition}\label{05prop:transmission regular QEC=delta2}
Let $G$ be a transmission regular grpah
and $\delta_2(G)$ denote the second largest eigenvalue of
the distance matrix $D$.
Then we have
\[
\mathrm{QEC}(G)=\delta_2(G).
\]
\end{proposition}

\begin{remark}\normalfont
For a general connected graph $G$, we have merely the inequality
$\delta_2(G)\le \mathrm{QEC}(G)<\delta_1(G)$.
It is an interesting open question to determine the class of
graphs $G$ such that $\mathrm{QEC}(G)=\delta_2(G)$.
This class contains all transmission regular graphs 
by Proposition \ref{05prop:transmission regular QEC=delta2}
and more.
For example, a path $P_n$ with an even $n\ge2$
is not transmission regular but
satisfies $\mathrm{QEC}(P_n)=\delta_2(P_n)$,
see \cite{Mlotkowski2022}.
\end{remark}

Now we go back to strongly regular graphs.

\begin{proposition}\label{05prop:srg is transmission regular}
Let $G$ be 
a strongly regular graph $\mathrm{srg}(n,k,\lambda,\mu)$ with $\mu\ge1$.
Then $G$ is transmission regular and
\[
D\bm{1}=\delta_1(G)\bm{1},
\qquad
\delta_1(G)=2n-k-2.
\]
\end{proposition}

\begin{proof}
By Proposition \ref{03prop:D=2J-2I-A} we have
\begin{equation}\label{03eqn:row sum distance matrix}
\sum_{y\in V} (D)_{xy}
=2n-2-\sum_{y\in V} (A)_{xy}
=2n-2-k,
\end{equation}
which means that $G$ is transmission regular and 
$\delta_1(G)=2n-k-2$.
\end{proof}

We now prove the main results 
(Theorerms \ref{01thm:main formula}--\ref{01thm:QEC<0}).

\begin{proof}[Proof of Theorem \ref{01thm:main formula}]
Since $G=\mathrm{srg}(n,k,\lambda,\mu)$ is a strongly regular graph
with $\mu\ge1$,
we see from Propositions 
\ref{05prop:transmission regular QEC=delta2}
and 
\ref{05prop:srg is transmission regular}
that $\mathrm{QEC}(G)=\delta_2(G)$.
On the other hand, it follows from 
Proposition \ref{03prop:eigen(r)<eigen(s)<eigen(k)} that
$\delta_2(G)=-s-2$, where $s$ is the minimal eigenvalue 
of the adjacency matrix $G$.
Finally, the formula
\begin{equation}\label{05eqn:main formula}
\mathrm{QEC}(G)
=-s-2
=-2-\frac{\lambda-\mu-\sqrt{(\lambda-\mu)^2+4(k-\mu)}}{2}
\end{equation}
follows from Proposition \ref{02prop:eigenvalues of A}.
\end{proof}

\begin{proof}[Proof of Theorem \ref{01thm:QEC le 0}]
In view of the formula in 
Theorem \ref{01thm:main formula}
or \eqref{05eqn:main formula}
we see that $\mathrm{QEC}(G)\le0$ if and only if 
\begin{equation}\label{04eqn:in proof Thm 4.5 (1)}
\sqrt{(\lambda-\mu)^2+4(k-\mu)}\le \lambda-\mu+4.
\end{equation}
Then, \eqref{04eqn:in proof Thm 4.5 (1)} is equivalent to the following
\begin{gather}
\lambda-\mu+4\ge0,
\label{04eqn:in proof Thm 4.5 (2)} \\
(\lambda-\mu)^2+4(k-\mu)\le (\lambda-\mu+4)^2.
\label{04eqn:in proof Thm 4.5 (3)}
\end{gather}
By simple algebra, \eqref{04eqn:in proof Thm 4.5 (3)} becomes
\begin{equation}\label{04eqn:in proof Thm 4.5 (4)}
k-2\lambda+\mu\le4.
\end{equation}
Note that \eqref{04eqn:in proof Thm 4.5 (2)} follows from
\eqref{04eqn:in proof Thm 4.5 (4)}.
In fact, from \eqref{04eqn:in proof Thm 4.5 (4)} we have
\[
\lambda-\mu+4 \ge k-\lambda \ge2,
\]
where Proposition \ref{02prop:bounds of parameters} is taken into account.
Finally, it is readily shown that
$\mathrm{QEC}(G)=0$ if and only if $k-2\lambda+\mu=4$.
\end{proof}

\begin{proof}[Proof of Theorem \ref{01thm:QEC<0}]
Suppose first that $2k+(n-1)(\lambda-\mu)=0$.
By Proposition \ref{02prop:conference graph} and 
the formula in 
Theorem \ref{01thm:main formula}
or \eqref{05eqn:main formula}
we see that $n\equiv 1 (\mathrm{mod}\,4)$ and
\[
\mathrm{QEC}(G)=-2-s=\frac{\sqrt{n}-3}{2}\,.
\]
Hence $\mathrm{QEC}(G)<0$ if and only if $n=5$,
that is, $G=C_5=\mathrm{srg}(5,2,0,1)$.

We recall \cite{Baskoro-Obata2021} that
$\mathrm{QEC}(G)\ge-1$ for any connected graph $G$
and the equality holds only when $G$ is a complete graph $K_n$
with $n\ge2$.
Now consider a strongly regular graph $G=\mathrm{srg}(n,k,\lambda,\mu)$
with $\mu\ge1$ and $2k+(n-1)(\lambda-\mu)\neq0$.
It follows from Proposition \ref{02prop:k,r,s are integers} 
that $s$ is an integer, so is $\mathrm{QEC}(G)$.
Therefore, $\mathrm{QEC}(G)<0$ means that $\mathrm{QEC}(G)=-1$
and $G$ is a complete graph.
But this contradict $G=\mathrm{srg}(n,k,\lambda,\mu)$
with $\mu\ge1$.
\end{proof}

During the above proof we have already established the following
result.

\begin{corollary}
Let $G$ be a strongly regular graph $\mathrm{srg}(n,k,\lambda,\mu)$
with $\mu\ge1$.
If $2k+(n-1)(\lambda-\mu)\neq0$,
then $\mathrm{QEC}(G)\ge0$ is a non-negative integer.
\end{corollary}

\section{Examples}
\label{06sec:Examples}

We collect some examples of strongly regular graphs with their
QE constants.
The computation is simple application 
of Theorem \ref{01thm:main formula}.

\begin{example}[cycles]\normalfont
Among cycles $C_n$ with $n\ge3$ only three are strongly regular graphs:
$C_3=K_3=\mathrm{srg}(3,2,1,*)$, 
$C_4=\mathrm{srg}(4,2,0,2)$ and $C_5=\mathrm{srg}(5,2,0,1)$.
Therefore, only $C_4$ and $C_5$ are strongly regular graphs:
$\mathrm{srg}(n,k,\lambda,\mu)$ with $\mu\ge1$.
We obtain
\[
\mathrm{QEC}(C_4)=0,
\qquad
\mathrm{QEC}(C_5)=\frac{-3+\sqrt{5}}{2}<0.
\]
The explicit value of $\mathrm{QEC}(C_n)$ for $n\ge3$ is known
\cite{Obata-Zakiyyah2018}.
Here we only recall that $\mathrm{QEC}(C_n)=0$ for an even $n\ge4$
and $\mathrm{QEC}(C_n)<0$ for an odd $n\ge3$.
\end{example}

\begin{example}[complete multipartite graphs]\normalfont
Let $p\ge2$ and $q\ge2$.
Let $K_{p\times q}=K_{q,q,\dots,q}$ ($q$ is repeated $p$ times) be
the complete multi-partite graph which consists of
$p$ islands and each island consists of $q$ vertices.
It is noted that $K_{p\times q}=\overline{pK_q}$, the complement of
the disjoint union of $p$ copies of the complete graph $K_q$.
We see that $K_{p\times q}$ is a strongly regular graph
with $\mathrm{srg}(pq,(p-1)q,(p-2)q,(p-1)q)$ and
the QE constant is given by
\[
\mathrm{QEC}(K_{p\times q})=q-2.
\]
In particular, $\mathrm{QEC}(K_{p\times2})=0$ for $p\ge2$.
The complete multipartite graph $K_{p\times 2}=\overline{pK_2}$ is 
also called the \textit{$p$-cocktail party graph}.
A formula for the QE constant of a general
complete multipartite graph $K_{m_1,m_2,\dots,m_k}$ is 
obtained in \cite{Obata2023b}.
\end{example}

\begin{example}[Conference graphs]\normalfont
A strongly regular graph 
is called a conference graph if its
parameter is of the form 
$\mathrm{srg}(n, (n-1)/2, (n-5)/4, (n-1)/4)$,
where $n\ge5$ and $n\equiv 1 (\mathrm{mod}\,4)$,
see Proposition \ref{02prop:conference graph}.
We have
\[
\mathrm{QEC}(G)=\frac{-3+\sqrt{n}}{2}\,.
\]
In particular, we see that 
all conference graphs are of non-QE class
except two cases $\mathrm{srg}(5,2,0,1)=C_5$ and
$\mathrm{srg}(9,4,1,2)$.
\end{example}

\begin{example}[line graphs of $K_n$]\normalfont
For $n\ge2$ the line graph of $K_n$, denoted by $L(K_n)$,
is a strongly regular graph with $\mathrm{srg}(n(n-1)/2,2(n-2),n-2,4)$.
We have 
\[
\mathrm{QEC}(L(K_n))=0, \qquad n\ge2.
\]
For $n=8$ we see that $L(K_8)$ has a parameter
$\mathrm{srg}(28, 12, 6,4)$,
which is shared with the Chang graphs,
see Example \ref{06ex:named graphs}.
It is also noted that the Petersen graph
is the complement of $L(K_5)$.
\end{example}

\begin{example}[line graphs of $K_{n,n}$]\normalfont
For $n\ge2$ the line graph of $K_{n,n}$, denoted by $L(K_{n,n})$,
is a strongly regular graph with $\mathrm{srg}(n^2,2n-2,n-2,2)$.
We have 
\[
\mathrm{QEC}(L(K_{n,n}))=0, \qquad n\ge2.
\]
The line graph $L(K_{n,n})$ is also called
a $n\times n$ square rook's graph.
\end{example}

\begin{example}[Some named graphs]\normalfont
\label{06ex:named graphs}
\,
\begin{center}
\begin{tabular}{|l|cccc|c|}\hline
graphs & $n$ & $k$ & $\lambda$ & $\mu$ & QEC \\
\hline
Petersen & 10 & 3 & 0 & 1 & 0
\\ \hline
Clebsch & 16 & 5 & 0 & 2 & 1
\\ \hline
Shrikhande & 16 & 6 & 2 & 2 & 0
\\ \hline
Schl\"afli & 27 & 16 & 10 & 8 & 0
\\ \hline
Changs & 28 & 12 & 6 & 4 & 0
\\ \hline
Hoffman-Singleton & 50 & 7 & 0 & 1 & 1
\\ \hline
Sims-Gewirtz & 56 & 10 & 0 & 2 & 2
\\ \hline
Brouwer-Haemers & 81 & 20 & 1 & 6 & 5
\\ \hline
Higman-Sims & 100 & 22 & 0 & 6 & 6
\\ \hline
\end{tabular}
\end{center}
\end{example}


\end{document}